\documentclass{amsart}  
\usepackage{amssymb, tikz, float, enumitem, mathtools, afterpage}
\usetikzlibrary{calc,decorations.markings,matrix,arrows,positioning}
\tikzset{>=latex}

\newcommand{\A}{\mathbb{A}}

\theoremstyle{plain} \newtheorem{thm}{Theorem}[section]
\newtheorem{prop}[thm]{Proposition}
\newtheorem{lem}[thm]{Lemma}
\newtheorem{cor}[thm]{Corollary}

\theoremstyle{definition} \newtheorem{defn}[thm]{Definition}
\theoremstyle{remark} 
\theoremstyle{plain} \newtheorem{claim}{Claim}
\theoremstyle{plain} 

\newenvironment{claimproof} {
  \begin{proof}[Proof of claim]
  
  } {
  \end{proof}
  }

\DeclareMathOperator{\Con}{Con}

\numberwithin{equation}{section}  
\renewcommand{\phi}{\varphi}
\renewcommand{\epsilon}{\varepsilon}

\usepackage{caption, subcaption, xparse}

\usepackage{framed}

\theoremstyle{definition} 
\theoremstyle{remark} 

\DeclareMathOperator{\cube}{cube}

\DeclareMathOperator{\faces}{Faces}
\DeclareMathOperator{\glue}{Glue}

\DeclareMathOperator{\tc}{TC}
\DeclareMathOperator{\cut}{Cut}

\DeclareMathOperator{\lines}{Lines}
\DeclareMathOperator{\squares}{Squares}
\DeclareMathOperator{\rect}{Rect}
\DeclareMathOperator{\refl}{Refl}

\DeclareMathOperator{\sym}{Sym}

\DeclareMathOperator{\rot}{rot}
\DeclareMathOperator{\com}{Com}
\DeclareMathOperator{\fin}{fin}

\DeclareMathOperator{\clo}{Clo}

\def\nat{\mathbb{N}}
\def\A{\mathbb{A}}
\def\var{\mathcal{V}}

\def\Union{\bigcup}

\def\union{\cup}

\def\finsub{\subseteq_{\fin}}

\tikzset{myStyle/.style={baseline=(center.base), font=\small,
    every node/.style={inner sep=0.25em} }}

\NewDocumentCommand{\LinePic}{ O{} O{} O{1} }{ 
  \begin{tikzpicture}[myStyle, scale=#3*1 ]
    \node (center) at (0,0.5) {\phantom{$\cdot$}}; 
    \path (0,0)  node (s) {$#1$}
        ++(0,1)  node (n) {$#2$};
    \draw (n) -- (s);
  \end{tikzpicture}
}  

\newcommand{\SquareUnwrapped}[4]{ 
  \node (center) at (0.5,-0.5) {\phantom{$\cdot$}}; 
  \path (0,0)  node (nw) {$#2$}
      ++(1,0)  node (ne) {$#4$}
      ++(0,-1) node (se) {$#3$}
      ++(-1,0) node (sw) {$#1$};
  \draw (nw) -- (ne) -- (se) -- (sw) -- (nw);
}  
\NewDocumentCommand{\SquareXY}{ O{} O{} O{} O{} O{1} O{1} }{ 
  \begin{tikzpicture}[myStyle, xscale=#5*1, yscale=#6*1 ]
    \SquareUnwrapped{#1}{#2}{#3}{#4}
  \end{tikzpicture}
}  
\NewDocumentCommand{\Square}{ O{} O{} O{} O{} O{1} }{ 
  \SquareXY[#1][#2][#3][#4][#5][#5]
}  

\NewDocumentCommand{\SquareAxes}{ O{} O{} O{1} O{0} O{1} }{  
  \begin{tikzpicture}[myStyle, scale=#3*0.8] 
    \node (center) at (0.5,0.5) {\phantom{$\cdot$}}; 
    \draw (0,0) -- node[above]{$#1$} (1,0) node[right]{$#4$}
      (0,0) -- node[left]{$#2$} (0,1) node[above]{$#5$};
  \end{tikzpicture}
}  
\NewDocumentCommand{\CubeAxes}{ O{} O{} O{} O{} O{} O{} O{} }{  
  \begin{tikzpicture}[myStyle, scale=#4*0.85] 
    \node (center) at (0.5,0.75) {\phantom{$\cdot$}}; 
    \draw (0,0) -- node[above]{$#1$} (1,0) node[right]{$#5$}
      (0,0) -- node[left]{$#2$} (0,1) node[above]{$#6$}
      (0,0) -- node[below left=-0.25em]{$#3$} (0.5,-0.5) node[below right=-0.2em]{$#7$};
  \end{tikzpicture}
}  

\newcommand{\CubeNodes}[8]{  
  \node at (0.75,-0.75) (center) {\phantom{$\cdot$}}; 
  \path (0,0)  node (back_nw)      {$#2$}
      ++(1,0)  node (back_ne)      {$#4$}
      ++(0,-1) node (back_se)      {$#3$}
      ++(-1,0) node (back_sw)      {$#1$}
        (0.5,-0.5) node (front_nw) {$#6$}
      ++(1,0)  node (front_ne)     {$#8$}
      ++(0,-1) node (front_se)     {$#7$}
      ++(-1,0) node (front_sw)     {$#5$};
}  
\newcommand{\CubeUnwrapped}[8]{ 
  \CubeNodes{#1}{#2}{#3}{#4}{#5}{#6}{#7}{#8}
  \draw (back_nw) -- (back_ne) -- (back_se) -- (back_sw) -- (back_nw)
    (front_nw) -- (front_ne) -- (front_se) -- (front_sw) -- (front_nw)
    (back_nw) -- (front_nw)
    (back_ne) -- (front_ne)
    (back_se) -- (front_se)
    (back_sw) -- (front_sw);
}  
\newcommand{\CubeDUnwrapped}[8]{ 
  \CubeNodes{#1}{#2}{#3}{#4}{#5}{#6}{#7}{#8}
  \draw (front_nw) -- (front_ne) -- (front_se) -- (front_sw) -- (front_nw)
    (back_nw) -- (back_ne) (back_sw) -- (back_nw)
    (back_nw) -- (front_nw)
    (back_ne) -- (front_ne)
    (back_sw) -- (front_sw);
  \draw[densely dotted] (back_ne) -- (back_se) -- (back_sw)
    (back_se) -- (front_se);
}  
\NewDocumentCommand{\Cube}{ O{} O{} O{} O{} O{} O{} O{} O{} O{1} }{  
  \begin{tikzpicture}[myStyle, scale=#9*1 ]
    \CubeUnwrapped{#1}{#2}{#3}{#4}{#5}{#6}{#7}{#8}
  \end{tikzpicture}
}  

\NewDocumentCommand{\CubeDeep}{ O{} O{} O{} O{} O{} O{} O{} O{} O{1}  }{  
  \begin{tikzpicture}[myStyle, xscale=#9*1, yscale=1.5 ]
    \CubeUnwrapped{#1}{#2}{#3}{#4}{#5}{#6}{#7}{#8}
  \end{tikzpicture}
}  
\NewDocumentCommand{\CubeD}{ O{} O{} O{} O{} O{} O{} O{} O{} O{1} }{  
  \begin{tikzpicture}[myStyle, scale=#9*1 ]
    \CubeDUnwrapped{#1}{#2}{#3}{#4}{#5}{#6}{#7}{#8}
  \end{tikzpicture}
}  

\NewDocumentCommand{\DeltaZeroCubeD}{ O{} O{} O{} O{} O{} O{} O{} O{} O{1} }{  
  \begin{tikzpicture}[myStyle, scale=#9*1]
    \CubeDUnwrapped{#1}{#2}{#3}{#4}{#5}{#6}{#7}{#8}
    \draw (back_sw)  to[out=30,in=180-30] (back_se)
      (back_nw)  to[out=30,in=180-30] node[auto]{$\delta$} (back_ne)
      (front_sw) to[out=30,in=180-30] (front_se);
    \draw[dashed] (front_nw) to[out=30,in=180-30] (front_ne);
  \end{tikzpicture}
}  

\NewDocumentCommand{\DeltaOneCubeD}{ O{} O{} O{} O{} O{} O{} O{} O{} O{1} }{  
  \begin{tikzpicture}[myStyle, scale=#9*1]
    \CubeDUnwrapped{#1}{#2}{#3}{#4}{#5}{#6}{#7}{#8}
    \draw (back_sw)  to[out=120,in=240]node[left]{$\delta$} (back_nw)
      (back_se)  to[out=120,in=240]  (back_ne)
      (front_sw) to[out=120,in=240] (front_nw);
  \end{tikzpicture}
}  
\NewDocumentCommand{\DeltaTwoCubeD}{ O{} O{} O{} O{} O{} O{} O{} O{} O{1} }{  
  \begin{tikzpicture}[myStyle, scale=#9*1]
    \CubeDUnwrapped{#1}{#2}{#3}{#4}{#5}{#6}{#7}{#8}
    \draw (back_sw) to[out=180+30,in=180] node[auto,swap]{$\delta$} (front_sw)
      (back_se) to[out=0,in=30] (front_se)
      (back_nw) to[out=180+30,in=180] (front_nw);
    \draw[dashed] (back_ne) to[out=0,in=30] (front_ne);
  \end{tikzpicture}
}  

\begin{document}
\title{Higher Kiss terms}
\author{ Andrew Moorhead}

\address[Andrew Moorhead]{
  University of Kansas
  Dept.\ of Electrical Engineering and Computer Science;
  Eaton Hall;
  Lawrence, KS 66044;
  U.S.A.}
\email[Andrew Moorhead]{apmoorhead@gmail.com}

\thanks{Andrew Moorhead has received funding from the European Research Council
(ERC) under the European Unions Horizon 2020 research and
innovation programme (grant agreement No 771005).
}

\begin{abstract}
We show that the modular term condition higher commutator is equal to the modular hypercommutator. As a consequence, we arrive at a new proof that HC8 holds for modular varieties. Next, we develop a procedure for a modular variety for producing the higher dimensional congruences that characterize the hypercommutator. This procedure allows us to demonstrate that every modular variety has an infinite sequence of what we call higher dimensional Kiss terms. We use these results to extend the scope of a theorem of Opr{\v s}al from permutable varieties to modular varieties. 
\end{abstract} 
\maketitle 

\section{Introduction}

The topic of this paper is the higher commutator. First defined by Bulatov in \cite{buldef}, the higher commutator is a congruence lattice operation that is a higher arity generalization of the binary commutator. Our overarching goal is to determine the extent to which the well-developed theory of the binary commutator (see for instance \cite{FM}, \cite{gumm}, and \cite{kearneskiss}) can be found as a low-dimensional case inside of a general multidimensional commutator theory. The results of this article demonstrate that there is indeed a higher dimensional commutator theory for modular varieties which extends the essential parts of Chapter 4 of \cite{FM} along with the properties of the Kiss term defined in \cite{threeremarks}.

The first results of this kind were proved by Aichinger and Mudrinski \cite{aichmud}. They demonstrate that, for congruence permutable varieties, many properties of the binary modular commutator are true also of the higher commutator. They also show that the higher commutator for a permutable variety satisfies an inequality relating nested commutator terms, which they refer to as HC8. 

Opr{\v s}al gives alternative proofs of these properties for permutable varieties in \cite{orsalrel}. His development of the $(n)$-ary commutator flows from the properties of $(2^n)$-ary relations in permutable varieties, usually called matrices, that satisfy higher dimensional versions of reflexivity and symmetry. He also defines a sequence of terms, each derived from a Mal'cev operation, which may be used along with knowledge of the higher commutator to recursively generate all of the matrices for an algebra in a permutable variety. He goes on to deduce that the collection of clones on a set sharing congruences, higher commutators, and a Mal'cev operation has a greatest element. 

In this article we show that the same kind of result is true for modular varieties, except that the relations we consider are also required to satisfy a higher dimensional version of transitivity and Day terms fill the role of a Mal'cev operation. This is the first of two articles that expand on the theory developed in \cite{taylorsupnil}. There, we define a new operation which we call the hypercommutator. The term condition higher commutator is defined with respect to what we call a higher dimensional tolerence, while the hypercommutator is defined with respect to what we call a higher dimensional congruence. 

The structure of this article is as follows. In Section \ref{sec:hyper=tc}, we show that the modular term condition higher commutator is equal to the modular hypercommutator. We derive as a consequence an alternative proof of HC8 for modular varieties, first shown to hold by Wires in \cite{Wires}. In Section \ref{sec:hdkt} we define the higher dimensional Kiss terms and establish some of their basic properties, and in Section \ref{sec:application} we extend Opr{\v s}al's Theorem 1.3 of \cite{orsalrel} from permutable varieties to modular varieties. We rely on the notation and theory developed in \cite{taylorsupnil} and cite results from \cite{FM}, so the reader should have both of these on hand.

\section{The hypercommutator for a modular variety}\label{sec:hyper=tc}

\subsection{A relational characterization of the term condition modular higher commutator}
\cite{taylorsupnil} 

In this section we argue that the term condition commutator is equal to the hypercommutator in a modular variety. This can be used to produce alternative proofs of HC8 for modular varieties and that the modular commutator is equal to the modular two-term commutator. The proof combines machinery from \cite{moorheadHC} and \cite{taylorsupnil}. To make the presentation consistent, we will first build a more general version of the machinery from \cite{moorheadHC} using the notation from \cite{taylorsupnil}. 

\begin{defn}\label{def:modularshiftrotations}
Let $\var$ be a modular variety with Day terms $m_0, \dots, m_{k}$ (see Theorem 2.2 of \cite{FM}). Let $\A \in \var$ and $n \geq 2$. For each $ e \in k+1$ and $ i \neq j \in n$, we define the \textbf{$e$th shift rotation at $(i,j)$} as 
\begin{align*}
\rot^e_{i,j}: A^{2^n } &\to A^{2^n}\\
			\gamma &\mapsto m_e\left( \refl_j^1(\gamma). \gamma, \refl_i^0(\gamma), \refl_j^1(\refl_i^0(\gamma))\right)
\end{align*}
\end{defn}
While the notation for a shift rotation used here does not specify the dimension of the relation on which it acts, the dimension should be clear from the context (see \emph{(4)} of Lemma \ref{lem:basicrotationproperties}).
\begin{lem}[cf.\ 3.1 of \cite{moorheadHC}]\label{lem:basicrotationproperties} Let $\var$ be a modular variety with Day terms $m_0, \dots, m_{k}$. Let $\A \in \var$, $\delta \in \Con(\A)$, $n \geq 2$, and $ i\neq j \in n$. The following hold:

\begin{enumerate}

\item If $R$ is an $(n)$-dimensional tolerance of $\A$, then 
$
\rot^e_{i,j} : R \to R,
$
for all $e \in k+1$.

\item Let $\gamma \in A^{2^n}$, $f \in 2^{n \setminus \{i, j \}}$, and suppose the $(i)$-supporting line of  $\squares_{i,j}(\gamma)_f$ is a $\delta$-pair. The $(i)$-pivot line of $\squares_{i,j}(\gamma)_f$ is a $\delta$-pair if and only if the $(j)$-pivot line of $\squares_{i,j}(\rot^e_{i,j}(\gamma))_f$ is a $\delta$-pair for every $e \in k+1$.

\item If $\gamma \in A^{2^n} $ is such that every $(i)$-supporting line of $\gamma$ is a $\delta$-pair, then every $(j)$-supporting line of $\rot^e_{i,j}(\gamma)$ is a $\delta$-pair for every $e \in k+1$. The same statement holds if `supporting' is replaced with `cross section'.

\item For  $ i \neq j \neq l \in n$, 
\[\rot_{i,j}^e(\gamma) = 
\glue_{\{l\}} \left( \rot^{e}_{i,j}( \faces_{l}^0(\gamma)), \rot^{e}_{i,j}( \faces_{l}^1(\gamma)) \right).
\]
\end{enumerate}

\end{lem}
\begin{proof} Assume that $R$ is an $(n)$-dimensional tolerance of $\A$ and take $\gamma \in R$. Because $R$ is $(n)$-reflexive, it contains  $\refl_j^1(\gamma), \refl_i^0(\gamma)$, and $\refl^1_j(\refl_i^0(\gamma))$. Because $R$ is closed under the operations of $\A$, it follows that \emph{(1)} holds. 

To see that the remaining properties hold, notice that if $\gamma \in A^{2^n} $ is such that 
\begin{align*}
\squares_{i,j}(\gamma) &=
\left\{
\begin{tikzpicture}
    [ baseline=(center.base), font=\small,
      every node/.style={inner sep=0.25em}, scale=1 ]
    \node at (0.5,-0.5) (center) {\phantom{$\cdot$}}; 
    \path (0,0)   node (nw) {$c_f$}
      -- ++(1,0)  node (ne) {$d_f$}
      -- ++(0,-1) node (se) {$b_f$}
      -- ++(-1,0) node (sw) {$a_f$};
    \draw (nw) -- (ne) -- (se) -- (sw) -- (nw);
\end{tikzpicture}
: f \in 2^{n \setminus \{ i,j \}}
\right\}, \text{ then for each $e \in k+1$}\\
\squares_{i,j}(\rot^e_{i,j}(\gamma)) &=
\left\{
\begin{tikzpicture}
    [ baseline=(center.base), font=\small,
      every node/.style={inner sep=0.25em}, scale=1 ]
    \node at (0.5,-0.5) (center) {\phantom{$\cdot$}}; 
    \path (0,0)   node (nw) {$c_f= m_e(c_f,c_f, c_f, c_f)$}
      -- ++(3,0)  node (ne) {$m_e(d_f,d_f, c_f, c_f)$}
      -- ++(0,-1) node (se) {$m_e(d_f,b_f, a_f, c_f)$}
      -- ++(-3,0) node (sw) {$c_f= m_e(c_f,a_f, a_f, c_f)$};
    \draw (nw) -- (ne) -- (se) -- (sw) -- (nw);
\end{tikzpicture}
: f \in 2^{n \setminus \{ i,j \}}
\right\}.
\end{align*}
Notice that we have applied one of Day's identities to conclude that the $(j)$-supporting line of any $(i,j)$-cross section square of any $\rot_{i,j}^e(\gamma)$ is a constant pair. Using Lemma 2.3 from \cite{FM}, the reader may now easily verify that \emph{(2)} and \emph{(3)} hold.  

The final property \emph{(4)} asserts that taking a shift rotation of an $(n)$-cube $\gamma$ commutes with projecting $\gamma$ onto a lower dimensional face, provided the image of the projection still depends on the coordinates $i$ and $j$. This is a consequence of Definition \ref{def:modularshiftrotations}.

\end{proof}

We have the following immediate consequence of the properties of the shift rotations.

\begin{prop}[cf.\ 3.2 of \cite{moorheadHC} ]\label{prop:modularcentralityissymmetric}
Let $\var$ be a modular variety, $\A \in \var$, and $n \geq 2$. Let $R$ be an $(n)$-dimensional tolerance of $\A$ and let $\delta \in \Con(\A)$. If $R$ has $(\delta, j)$-centrality for some $j \in n$, then $R$ has $(\delta, i)$-centrality for each $i \in n$. 

\end{prop}

\begin{proof}
Suppose that the assumptions hold and that $R$ has $(\delta, i)$-centrality. Let $ j \in n $ and take $\gamma \in R$ such that every every $(i)$-supporting line of $\gamma$ is a $\delta$-pair. If $i = j$ then it follows from our assumptions that the $(i)$-pivot line of $\gamma$ is a $\delta$-pair. If $i \neq j$ we apply \emph{(3)} then \emph{(2)} of Lemma \ref{lem:basicrotationproperties} and conclude that the $(i)$-pivot line of $\gamma$ is a $\delta$-pair. 
\end{proof}

In view of Proposition \ref{prop:modularcentralityissymmetric}, we may now assert with no chance for confusion that an $(n)$-dimensional tolerance of an algebra $\A$ belonging to a modular variety has $\delta$-centrality for some $\delta \in \Con(\A)$. 

As in both \cite{moorheadHC} and \cite{taylorsupnil}, we need to consider certain compositions of shift rotations and will use finite trees to index these compositions. Suppose that $\var$ is a modular variety with Day terms $m_0, \dots, m_k$ and let $n \geq 2$. Set

\[
\mathbb{D}_n = \langle  (k+1)^{ < n}; \leq \rangle,
\]
where $ (k+1)^{ < n} = \Union \{ (k+1)^i : i \in n \}$  and two sequences $d_1$, $d_2$ are $\leq$-related when $d_1 \subseteq d_2$. Note that $\mathbb{D}_n$ has the empty sequence $\emptyset$ as its root. For $\A \in \var$ and $\gamma \in A^{2^n}$, set $\gamma^\emptyset = \gamma$. We recursively define $\gamma^d = \rot_{i,i+1}^{d_i}(\gamma^c)$, where $d =(d_0, \dots , d_i) \in \mathbb{D}_{n}$ is non-empty and $c$ is the predecessor of $d$.

\begin{lem}[cf.\ 4.1 of \cite{moorheadHC}]\label{lem:treesuccessorrotations}
Let $\var$ be a modular variety with Day terms $m_0, \dots, m_k$. Let $\A \in \var$ and $R \leq \A^{2^n}$ be an $(n)$-dimensional tolerance for some $n\geq 2$. If $d \in \mathbb{D}_{n}$ is a tuple of length $i \in n$, then
\begin{enumerate}

\medskip
\item 
$\gamma^d \in R$, and 
\medskip

\item
if $f \in 2^{n \setminus \{i\}}$ satisfies $f(j) = 0$ for some $j\in i$, then the $(i)$-cross-section line of $\gamma^d$ at $f$
is a constant pair.
\end{enumerate}

\end{lem}
\begin{proof}
The result is trivially true for $\gamma^{\emptyset} = \gamma$. Suppose that it holds for $c \in \mathbb{D}_{n}$ of length $i$ and let $d=(d_0, \dots, d_i)$ be a successor of $c$. Set 
$
\gamma^d = \rot^{d_i}_{i,i+1}(\gamma^c). 
$
Notice that \emph{(1)} of Lemma \ref{lem:basicrotationproperties} guarantees that $\gamma^d \in R$, which establishes \emph{(1)} of this lemma.

Now let $f \in 2^{n\setminus{i+1}}$ be such that $f(j) = 0$ for some $j \in i+1$, and let $f^*$ be the restriction of $f$ to the set $n\setminus{\{i, i+1\}}$.  There are two cases to consider.

\begin{enumerate}
\item[Case 1:] Suppose $j \notin i$, in which case $f(i)=0$. If 
\[
\squares_{i,i+1}(\gamma^c)_{f^*} = \Square[a][e][b][d],
\]  
then it follows from the proof of Lemma \ref{lem:basicrotationproperties} that 
\[
\squares_{i,i+1}(\gamma^d)_{f^*} = 
 \begin{tikzpicture}
    [ baseline=(center.base), font=\small,
      every node/.style={inner sep=0.25em}, scale=1 ]
    \node at (0.5,-0.5) (center) {\phantom{$\cdot$}}; 
    \path (0,0)   node (nw) {$e= m_{d_i}(e,e, e, e)$}
      -- ++(3,0)  node (ne) {$m_{d_i}(d,d, e, e)$}
      -- ++(0,-1) node (se) {$m_{d_i}(d,b, a, e)$}
      -- ++(-3,0) node (sw) {$e= m_{d_i}(e,a, a, e)$};
    \draw (nw) -- (ne) -- (se) -- (sw) -- (nw);
\end{tikzpicture}.
\]
The left column of the above square is equal to the $(i+1)$-cross-section line of $\gamma^d$ at $f$ and it is a constant pair, as claimed.  
\item[Case 2:] Suppose $ j \in i$. In this case we apply the inductive assumption that \emph{(2)} holds for $\gamma^c$ and conclude that 
\[
\squares_{i, i+1}(\gamma^c)_{f^*} = \Square[a][e][a][e], \text{ so it follows that}
\]
 
\[
\squares_{i,i+1}(\gamma^d)_{f^*} = 
\begin{tikzpicture}
    [ baseline=(center.base), font=\small,
      every node/.style={inner sep=0.25em}, scale=1 ]
    \node at (0.5,-0.5) (center) {\phantom{$\cdot$}}; 
    \path (0,0)   node (nw) {$e= m_{d_i}(e,e, e, e)$}
      -- ++(3,0)  node (ne) {$m_{d_i}(e,e, e, e)=e$}
      -- ++(0,-1) node (se) {$m_{d_i}(e,a, a, e)=e$}
      -- ++(-3,0) node (sw) {$e= m_{d_i}(e,a, a, e)$};
    \draw (nw) -- (ne) -- (se) -- (sw) -- (nw);
\end{tikzpicture}.
\]
The $(i+1)$-cross-section line of $\gamma^d$ at $f$ is either the left column or right column of the above square, and each of these columns is a constant pair. This finishes the proof. 
\end{enumerate}  
\end{proof}

\begin{prop}[cf.\ 4.7 of \cite{taylorsupnil}]\label{prop:tchascentrality}
Let $\var$ be a modular variety with Day terms $m_0, \dots, m_k$. Let $\A \in \var$, $n \geq 2$, $\delta \in \Con(\A)$. Let $l \in n$. If $R \leq \A^{2^n}$ is an $(n)$-dimensional tolerance of $\A$ that has $\delta$-centrality, then $R^{\circ_l}$ is also an $(n)$-dimensional tolerance of $\A$ that has $\delta$-centrality. 
\end{prop}

\begin{proof}
Before we embark on the proof, we point out that the argument we give is an amalgamation of the proof that the modular commutator distributes over arbitrary joins (see 5.1 of \cite{moorheadHC}) and the proof of the `perpendicular stage' (see 4.7 of \cite{taylorsupnil}). Each of those proofs is accompanied by a picture that the reader may find useful for visualizing the argument. 

Without loss it suffices to consider the case when $l=n-1$. We noted in Lemma 2.9 of \cite{taylorsupnil} that $R^{\circ_{n-1}} $ is an $(n)$-dimensional tolerance of $\A$, so it remains to verify that $R^{\circ_{n-1}}$ has $\delta$-centrality whenever $R$ has $\delta$-centrality. By Proposition \ref{prop:modularcentralityissymmetric}, it is enough to show that $R^{\circ_{n-1}}$ has $(\delta, 0)$-centrality. 

So, let $\gamma \in R^{\circ_{n-1}}$ be such that every $(0)$-supporting line is a $\delta$-pair. Our task is to show that the $(0)$-pivot line of $\gamma$ is also a $\delta$-pair. By the definition of $R^{\circ_{n-1}}$, there are $\mu_0, \dots, \mu_{s-1} \in A^{2^{n-1}}$ so that 
\begin{enumerate}
\item $\faces_{n-1}^0(\gamma) = \mu_0$,
\item $ \faces_{n-1}^1(\gamma) = \mu_{s-1}$, and
\item $\glue_{\{n-1\}}(\langle \mu_r, \mu_{r+1} \rangle ) \in R$, for each 
$r \in s-1$.
\end{enumerate}

\begin{claim}\label{claim:perpclaim1}

Take $d \in \mathbb{D}_{n-1}$ to be a leaf. The sequence 
$
(\mu_0)^d, \dots, (\mu_{s-1})^d \in A^{2^{n-1}}
$
satisfies
\begin{enumerate}
\item for all $r \in s$, each $(n-2)$-supporting line of $(\mu_r)^d$ is a $\delta$-pair,
\item every $(n-2)$-cross section line of $(\mu_0)^d $ is a $\delta$-pair, and
\item $\glue_{\{n-1\}}(\langle (\mu_r)^d, (\mu_{r+1})^d \rangle) \in R$,
 for all 
$r \in s-1$.

\end{enumerate}
\end{claim}

\begin{claimproof}
Suppose $d = (d_0, \dots, d_{n-3})$. The first property of the claim follows from \emph{(2)} of Lemma \ref{lem:treesuccessorrotations} and the fact that $\delta$ contains all constant pairs. 
To show the second property of the claim, we proceed by induction over the branch in $\mathbb{D}_{n-3}$ determined by $d = (d_0, \dots, d_{n-3})$. We assume $\mu_0 = \faces_{n-1}^0(\gamma)$ and that every $(0)$-supporting line of $\gamma$ is a $\delta$-pair. It follows that every $(0)$-cross section line of $\mu_0 = (\mu_0)^\emptyset$ is a $\delta$-pair, which establishes the basis of the induction. Now let $(d_0, \dots, d_i)$ be an ancestor of $d$ and suppose that every $(i+1)$-cross section line of $(\mu_0)^{(d_0, \dots, d_i)}$ is a $\delta$-pair. Now \emph{(3)} of Lemma \ref{lem:basicrotationproperties} shows that every $(i+2)$-cross section line of $(\mu_0)^{(d_0, \dots, d_{i+1})}$ is a $\delta$-pair. It follows that every $(n-2)$-cross section line of $(\mu_0)^d$ is a $\delta$-pair, as claimed. 

A similar induction using \emph{(1)} and \emph{(4)} of Lemma \ref{lem:basicrotationproperties} establishes the third property of the claim.

\end{claimproof}

\begin{claim}\label{claim:perpclaim2}
If $d \in \mathbb{D}_{n-1}$ is a leaf, then the $(n-2)$-pivot line of $(\mu_{r+1})^d$ is a $\delta$-pair for all $r \in s-1$.
\end{claim}
\begin{claimproof}
The claim follows by induction on $r \in  s$. The claim holds for $r=0$ by \emph{(2)} of Claim  \ref{claim:perpclaim1}. Suppose the claim holds for $\mu_r$ for $r \in s-2$. This assumption along with \emph{(1)} and \emph{(3)} of Claim \ref{claim:perpclaim1} show that every $(n-2)$-supporting line of
\[
\glue_{\{n-1\}}(\langle (\mu_{r+1})^d, (\mu_{r+2})^d \rangle) \in R
\]
is a $\delta$-pair. We now apply the assumption that $R$ has $\delta$-centrality and conclude that the $(n-2)$-pivot line of this cube is also a $\delta$-pair. Because 
\[\glue_{\{n-1\}}(\langle (\mu_{r+1})^d, (\mu_{r+2})^d \rangle)
\] and $(\mu_{r+2})^d$ have the same $(n-2)$-pivot line, the claim is proved. 
\end{claimproof}

\begin{claim}\label{claim:perpclaim3}
Let $c  \in \mathbb{D}_{n-1}$ be a tuple of length $z \in n-1$. The $(z)$-pivot line of $(\mu_{s-1})^c$ is a $\delta$-pair. In particular, the $(0)$-pivot line of 
$(\mu_{s-1})^\emptyset = \mu_{s-1}$ is a $\delta$-pair.
\end{claim}
\begin{claimproof}

We proceed by an induction from the leaves of $\mathbb{D}_{n-1}$ to its root. The basis has been established by Claim \ref{claim:perpclaim2}. Suppose that the claim holds for all tuples of length $z +1 \in n-1$ belonging to $\mathbb{D}_{n-1}$ and let $c \in \mathbb{D}_{n-1}$ be a tuple of length $z$. Our assumption that every $(0)$-supporting line of $\gamma$ is a $\delta$-pair implies that every $(0)$-supporting line of $\mu_{s-1}$ is a $\delta$-pair. An induction using \emph{(3)} of Lemma \ref{lem:basicrotationproperties} shows that every $(z)$-supporting line of $(\mu_{s-1})^c$ is a $\delta$-pair. In particular, the $(z)$-supporting line of the $(z,z+1))$-pivot square of $(\mu_{s-1})^c$ is a $\delta$-pair. The inductive assumption that the claim holds for all $d \in \mathbb{D}_{n-1}$ of length $z+1$ implies that the $(z+1)$-pivot line of $\rot_{z,z+1}^{e}((\mu_{s-1})^c)$ is a $\delta$-pair for every $e \in k+1$. It follows that \emph{(2)} of Lemma \ref{lem:basicrotationproperties} may be applied, and we conclude that the $(z)$-pivot line of $(\mu_{s-1})^c$ is a $\delta$-pair.

\end{claimproof}

The $(0)$-pivot lines of $\gamma$ and $\mu_{s-1}$ are the same, and the conclusion of Claim \ref{claim:perpclaim3} is that the $(0)$-pivot line of $\mu_{s-1}$ is a $\delta$-pair. This is what we wanted to show, so the proof is finished.

\end{proof}

\begin{thm}\label{thm:tc=hhcinmodular}
If $\var$ is a modular variety, then the $(n)$-ary term condition commutator for $\var$ is equal to the $(n)$-ary hypercommutator for $\var$.
\end{thm}

\begin{proof}
Let $\var$ be modular and let $\A \in \var$. Take $n \geq 2$ and pick some $(\theta_0, \dots, \theta_{n-1}) \in \Con(\A)^n$. Our task is to show that

\[
[\theta_0, \dots, \theta_{n-1}]_{TC} = [ \theta_0, \dots, \theta_{n-1}]_{H}.
\]
The hypercommutator is always bigger than the term condition commutator, so it suffices to show that $\Delta(\theta_0, \dots,\theta_{n-1})$ has $\delta$-centrality, where $\delta = [\theta_0, \dots, \theta_{n-1}]_{TC}$ (see 2.17 of \cite{taylorsupnil}). This follows from an induction using Proposition \ref{prop:tchascentrality}. Indeed, 

\[
\Delta(\theta_0, \dots, \theta_{n-1}) = \tc(M(\theta_0, \dots, \theta_{n-1})),
\]
and $M(\theta_0, \dots, \theta_{n-1})$ has $\delta$-centrality by definition. 
\end{proof}

The theory of the hypercommutator for any variety includes HC8 as a true statement (we denote this fact by HHC8, see 4.13 of \cite{taylorsupnil}), so the following corollary is immediate. 

\begin{cor}
HC8 is a true statement belonging to the theory of the modular higher commutator.
\end{cor}

\subsection{The $\Delta$ relation in a modular variety}

In this section we develop a method for generating the $\Delta$ relation from lower dimensional $\Delta$ relations in the case that the algebra under investigation belongs to a modular variety. As we will see later, this characterization of $\Delta$ is important for the computations that demonstrate the existence of the `higher dimensional' Kiss terms for a modular variety.  

We first look at the relationship between the definition of a $(2)$-dimensional congruence and the development of the binary modular commutator in \cite{FM}. For this paragraph let $\A$ be an algebra belonging to a modular variety and let $\alpha, \beta \in \Con(\A)$. In Chapter 4 the authors define the relation $\Delta_{\beta, \alpha}$ to be the least congruence of $\beta$ that collapses any two diagonal pairs whose values are $\alpha$-related. The authors use $2\times2$ matrices to represent pairs in this relation, with the convention that rows or columns determine $\alpha$ or $\beta$ pairs, respectively. The reader is asked to prove in exercise 5 of Chapter 4 that $\Delta_{\beta, \alpha}$ is a $(2)$-dimensional congruence (we are translating the statement using our language). This means that, in the modular setting, the $\Delta_{\beta, \alpha}$ defined in \cite{FM} is the same relation the $\Delta(\alpha, \beta)$ defined in \cite{taylorsupnil}.

For us the notion of a $(2)$-dimensional congruence is taken as a definition, so the ideas in \cite{FM} can be turned around to give sufficient conditions for a $(2)$-dimensional tolerance to be a $(2)$-dimensional congruence in a modular variety. Informally, the next lemma asserts that in a modular variety, any $(2)$-dimensional tolerance that is `almost' a $(2)$-dimensional congruence is a $(2)$-dimensional congruence. 

\begin{lem}\label{lem:almost2transitivegivestransitive}
Let $\var$ be a modular variety and let $\A \in \var$. Let $k \neq l \in \nat$ and suppose that $R \leq \A^{2^{\{k,l\}}}$ is a $(2)$-dimensional tolerance of $\A$. If $\faces_k(R) \in \Con(\faces_k^0(R))$ and $\faces_k^0(R) \in \Con(\A)$, then $R$ is a $(2)$-dimensional congruence. 
\end{lem}

\begin{proof}
Before we begin, we remark that this proof is essentially identical to the solution of exercise 5 of Chapter 4 of \cite{FM}. Because we have weakened the assumptions and changed the terminology, we reproduce the argument here for the convenience of the reader. 

In this proof every labeled square belonging to $R$ is oriented so that the $k$-direction is horizontal and the $l$-direction is vertical. Our goal is to show that $\faces_l(R)$ is a transitive relation. We begin by establishing the following special case. 

\begin{claim}\label{claim:shiftingclaim}
If 
$
\Square[a][c][a][d][.6], \Square[b][a][b][a][.6] \in R,
$
then $\Square[b][c][b][d][.6] \in R$.
\end{claim}
\begin{claimproof}
We denote by $\theta$ the congruence $\faces_k^0(R)$. The assumptions of the claim guarantee that $a,b,c$ and $d$ belong to the same $\theta$-class, so in particular 
\[
\LinePic[b][c][.6], \LinePic[b][d][.6] \in \theta.
\]
To demonstrate that these two pairs are related by $\faces_k(R)$, it suffices to apply the Shifting Lemma (see 2.4 of \cite{FM}) to the following diagram, where the $\eta_i$'s are projection kernels. 

\[
\begin{tikzpicture}
    [ baseline=(center.base), font=\small,
      every node/.style={inner sep=0.25em}, scale=1 ]
    \node at (0.5,-0.5) (center) {\phantom{$\cdot$}}; 
    \path (0,0)   node (nw) {$\LinePic[b][c][.6]$}
      -- ++(2,0)  node (ne) {$\LinePic[a][c][.6]$}
      -- ++(0,-1.5) node (se) {$\LinePic[a][d][.6]$}
      -- ++(-2,0) node (sw) {$\LinePic[b][d][.6]$};
    \draw (nw) -- node[above] {$\eta_0$}(ne) -- (se) -- (sw) -- node[right] {$\eta_1$} (nw);
    
    \draw (ne) edge[bend left] node[right] {$\faces_k(R)$} (se);
    
    \end{tikzpicture}
\]

\end{claimproof}
Moving on to the general case, suppose that 
$
\Square[u][x][v][y][.6], \Square[r][u][s][v][.6] \in R.
$
Suppose that $\var$ has Day terms $m_0, \dots, m_n$. We want to show that $\Square[r][x][s][y][.6] \in R$. By 2.3 of \cite{FM}, it is enough to see that  
\[
\begin{tikzpicture}
    [ baseline=(center.base), font=\small,
      every node/.style={inner sep=0.25em}, scale=1 ]
    \node at (0.5,-0.5) (center) {\phantom{$\cdot$}}; 
    \path (0,0)   node (nw) {$m_e(x,x,y,y)$}
      -- ++(2.2,0)  node (ne) {$m_e(x,x,y,y)$}
      -- ++(0,-1) node (se) {$m_e(r,u,v,s)$}
      -- ++(-2.2,0) node (sw) {$m_e(r,r,s,s)$};
    \draw (nw) -- (ne) -- (se) -- (sw) -- (nw);
    \end{tikzpicture}
    \in R,
\]
for every $e \in n+1$. The assumptions imply that both
\[
\LinePic[u][u][.6] \faces_k(R) \LinePic[v][v][.6] 
\qquad \text{and} \qquad
\LinePic[s][v][.6] \faces_k(R) \LinePic[r][u][.6],
\]
and it follows that 
\[
\LinePic[m_e(r,r,s,s)][m_e(u,u,v,v)][.6]
\faces_k(R)
\LinePic[m_e(r,r,r,r)][m_e(u,u,u,u)][.6]
= 
\LinePic[r][u][.6]
= 
\LinePic[m_e(r,v,v,r)][m_e(u,v,v,u)][.6]
\faces_k(R)
\LinePic[m_e(r,u,v,s)][m_e(u,u,v,v)][.6].
\]
Because $\faces_k(R)$ is transitive, we have shown that 
\newline
$
\begin{tikzpicture}
    [ baseline=(center.base), font=\small,
      every node/.style={inner sep=0.25em}, scale=1 ]
    \node at (0.5,-0.5) (center) {\phantom{$\cdot$}}; 
    \path (0,0)   node (nw) {$m_e(u,u,v,v)$}
      -- ++(2.2,0)  node (ne) {$m_e(u,u,v,v)$}
      -- ++(0,-1) node (se) {$m_e(r,u,v,s)$}
      -- ++(-2.2,0) node (sw) {$m_e(r,r,s,s)$};
    \draw (nw) -- (ne) -- (se) -- (sw) -- (nw);
    \end{tikzpicture} \in R.
$ Furthermore, it is easy to see that
\newline
$
\begin{tikzpicture}
    [ baseline=(center.base), font=\small,
      every node/.style={inner sep=0.25em}, scale=1 ]
    \node at (0.5,-0.5) (center) {\phantom{$\cdot$}}; 
    \path (0,0)   node (nw) {$m_e(x,x,y,y)$}
      -- ++(2.2,0)  node (ne) {$m_e(x,x,y,y)$}
      -- ++(0,-1) node (se) {$m_e(u,u,v,v)$}
      -- ++(-2.2,0) node (sw) {$m_e(u,u,v,v)$};
    \draw (nw) -- (ne) -- (se) -- (sw) -- (nw);
    \end{tikzpicture}
    \in R,
$
and the lemma follows by applying Claim \ref{claim:shiftingclaim}.

\end{proof}

We now argue that the previous lemma is actually a special case of a more general principle. Informally, this principle guarantees that in a modular variety, every $(n)$-dimensional tolerance that is `almost' an $(n)$-dimensional congruence is an $(n)$-dimensional congruence. More precisely, we have the following proposition.

\begin{prop}\label{prop:almostnconisncon}
Let $\var$ be a modular variety and let $\A \in \var$. Let $n \geq 2$ and suppose that $R \leq \A^{2^n}$ is an $(n)$-dimensional tolerance of $\A$ with the property that $\faces_i^0(R)$ is a $(n-1)$-dimensional congruence for each $i \in n$. If there exists $k\in n$ such that $\faces_k(R) \in \Con(\faces_k^0(R))$, then $R$ is an $(n)$-dimensional congruence. 
\end{prop}

\begin{proof}
Because $R$ is an $(n)$-dimensional tolerance, we just need to show that $R$ is $(n)$-transitive. Let $k$ be as in the proposition statement. We assume that $\faces_k(R)$ is transitive, so let $l \neq k \in n$. We use Lemma \ref{lem:almost2transitivegivestransitive} to deduce that $R$ is transitive in the $l$-direction from the transitivity of $R$ in the $k$-direction. We first ask the reader to verify the following easy claim.

\begin{claim}\label{claim:sufficienttochecksquares}
Let $s\neq t \in n$. 
 The relation $\faces_s(R) $ is transitive if and only if $\faces_s(\cut_{\{s,t\}}(R))$ is transitive. Similarly, $\faces_t(R) $ is transitive if and only if $\faces_t(\cut_{\{s,t\}}(R))$ is transitive. 
\end{claim}

Now, a typical element of  
$\cut_{\{k,l\}}(R)$ looks like

\[
\begin{tikzpicture}
    [ baseline=(center.base), font=\small,
      every node/.style={inner sep=0.25em}, scale=1 ]
    \node at (0.5,-0.5) (center) {\phantom{$\cdot$}}; 
    \path (0,0)   node (nw) {$\cut_{\{k,l\}}(\gamma)_{\{\langle k,0 \rangle, \langle l,1\rangle \}}$}
      -- ++(7,0)  node (ne) {$\cut_{\{k,l\}}(\gamma)_{\{\langle k,1 \rangle, \langle l,1\rangle \}}$}
      -- ++(0,-2) node (se) {$\cut_{\{k,l\}}(\gamma)_{\{\langle k,1 \rangle, \langle l,0\rangle \}}$}
      -- ++(-7,0) node (sw) {$\cut_{\{k,l\}}(\gamma)_{\{\langle k,0 \rangle, \langle l,0\rangle \}}$};
      \node at (3.5,-1) {$\leftarrow \faces_k(\cut_{\{k,l\}}(R)) \rightarrow$};
    \draw (nw) --  (ne) -- (se) -- node[below] {$\faces_k(\faces_l^0(R))$} (sw) -- node[left] {$\faces_l(\faces_k^0(R))$}(nw);
    \end{tikzpicture}.
\]
We can apply Lemma \ref{lem:almost2transitivegivestransitive}, with the algebra under consideration set to 
\newline
$\cut_{\{k,l\}}(R)_{\{\langle k, 0 \rangle , \langle l, 0 \rangle\}}$ and the $(2)$-dimensional relation under consideration set to $\cut_{\{k,l\}}(R)$. Claim \ref{claim:sufficienttochecksquares} ensures that $\faces_k(\cut_{\{k,l\}}(R))$ is transitive. It is similarly straightforward to see that the other assumptions of the lemma are satisfied, so we conclude that $\faces_l(\cut_{\{k,l\}}(R))$ is transitive. 

\end{proof}

\begin{thm}\label{thm:deltaofdeltaisdelta}
Let $\var$ be a modular variety. Let $\A \in \var$, $S \subseteq \nat$ with $|S| \geq 3$, and $\{\theta_k\}_{k \in S} \in \Con(\A)^S$. Let $Q \subseteq S$. If for each $i \in Q$ we set
\begin{align*}
\alpha_i &= \faces_i
\bigg(
\Delta
\left(
\{\theta_j\}_{j \in \{i\} \union S \setminus Q}
\right)
\bigg), \text{ then}\\
\Delta
\left(
\{ \theta_k \}_{k \in S}
\right)
 &= \glue_Q
\bigg(
\Delta
\left(
\{\alpha_i\}_{ i \in Q}
\right)
\bigg).
\end{align*}

\end{thm}

\begin{proof}
We should first explain how the theorem statement makes sense. First, notice that for each $i \in Q$,  
\[
\alpha_i \in 
\Con
\bigg(
\Delta
\left(
\{\theta_j\}_{j \in  S \setminus Q}
\right)
\bigg).
\]
It follows that 
\begin{align*}
\Delta(\{\alpha_i\}_{i \in Q}) &\subseteq
\left(
\Delta
\left(
\{\theta_j\}_{j \in  S \setminus Q}
\right)
\right)^Q
\subseteq (A^{2^{S \setminus Q}})^Q.
\end{align*}
Hence, $\Delta(\{\alpha_i\}_{i \in Q})$
is a subset of the domain of the mapping $\glue_{Q}$ and theorem's assertion contains no formal conflict. We also notice that if $|Q| \in \{ 0, 1, |S| \}$, then the theorem is trivial, and so we assume that we are not dealing with one of these cases. 

Set   
$\Delta = \Delta
\left(
\{ \theta_k \}_{k \in S}
\right)$ and 
$\Delta' = \glue_Q
\bigg(
\Delta
\left(
\{\alpha_i\}_{ i \in Q}
\right)
\bigg).
$
The result will follow after we notice that 
\[
M
\left(
\{\theta_k\}_{k \in S}
\right)
\subseteq
\Delta'
\subseteq \Delta
\]
and then demonstrate that $\Delta'$ is a $(|S|)$-dimensional congruence. Denote by $\cube(x)$ the vertex labeled cube with every vertex labeled by $x$. Note that the dimension of this cube depends on the context in which it is used. 

To show that 
$M
\left(
\{\theta_k\}_{k \in S}
\right)
\subseteq
\Delta'
$,
it is enough to see that $\Delta'$ contains its generating set. Let $k \in S$ and $\langle x,y \rangle \in \theta_k$. We have that 

\[
\cut_Q(\cube_k(x,y)) = 
\begin{cases}
\cube_k( \cube(x), \cube(y)) &\text{ if } k \in Q \\
\cube(\cube_k(x,y)) &\text{ if } k \in S \setminus Q.
\end{cases}
\]
In either case, $\cut_Q(\cube_k(x,y))$ is a generator of $\Delta\left( \{\alpha_i\})_{i \in Q} \right)$, hence $\cube_k(x,y) \in \Delta'$. 

To show that $\Delta' \subseteq \Delta$, we will demonstrate that 
$\glue_Q(M(\{\alpha_i\}_{i \in Q})) \subseteq \Delta$. This is adequate, because it has the consequence that
\[
\Delta' = \glue_Q(\tc
\left(
M(\{\alpha_i\}_{i \in Q})
\right)
\subseteq \tc
\left(
\glue_Q(
M(\{\alpha_i\}_{i \in Q})
)
\right)
\subseteq \Delta .
\]
Because term operations commute with $\glue_Q$, we need only show that $\glue_Q(\mu) \in \Delta$, where $\mu$ is a generator of $M(\{\alpha_i\}_{i \in Q})$. 
So, let $i \in Q$ and suppose $\mu = \cube_i( \langle \zeta, \eta \rangle)$, where $\zeta, \eta \in \Delta(\{\theta_j\}_{j \in S \setminus Q})$ and $\langle \zeta, \eta \rangle \in \alpha_i$. Now, Lemma 2.14 of \cite{taylorsupnil} indicates that $\cut_{Q\setminus \{i\}}(\Delta)_\textbf{1} = \Delta(\{\theta_j\}_{j \in \{i\} \union S\setminus Q})$. There is therefore a $\gamma \in \Delta$ satisfying 
$
\cut_{Q\setminus{\{i\}}}(\gamma)_\textbf{1}= \glue_{\{i\}}(\langle \zeta, \eta\rangle).
$
Applying Corollary 2.5 of \cite{taylorsupnil} to this situation produces a $\tau$ such that 
\[
\cut_{Q\setminus \{i\}}(\tau)_f = \glue_{\{i\}}(\langle \zeta, \eta\rangle)
\]
holds for every $f \in 2^{Q\setminus \{i\}}$. This exactly means that $\cut_Q(\gamma) = \mu$, or put another way, that $\glue_Q(\mu) = \gamma \in \Delta$. 

Now we apply Proposition \ref{prop:almostnconisncon} to show that $\Delta'$ is a $(|S|)$-dimensional congruence. Let us establish that the conditions are satisfied. We first show that $\Delta'$ is a $(|S|)$-dimensional tolerance. First, we note that $\Delta' \subseteq \tc(\glue_Q(M(\{\alpha_i\}_{i\in Q})))$, so it is enough to establish that $\glue_Q(M(\{\alpha_i\}_{i\in Q}))$ is a $(|S|)$-dimensional tolerance (this follows from Lemma 2.9 of \cite{taylorsupnil}).

Let $\gamma \in \Delta'$, $k \in S$ and $j \in 2$. We want to show that $\refl^j_k(\gamma), \sym_k(\gamma) \in \glue_Q(M(\{\alpha_i\}_{i\in Q}))$. It is sufficient to check this for those $\gamma = \glue_Q(\mu)$ for $\mu$ a generator. The case when $k \in Q$ is obvious, so suppose $k \in S\setminus Q$. Let $i \in Q$ and $\langle \zeta, \eta \rangle \in \alpha_i$ determine a generator $\mu =\cube_i(\zeta, \eta)$. It follows from the definition of $\alpha_i$ that $\langle \refl^j_k(\zeta), \refl^j_k(\eta) \rangle, \langle \sym_k(\zeta), \sym_k(\eta) \rangle  \in \alpha_i$, therefore $\mu^{r_j} = \cube_i(\refl^j_k(\zeta), \refl^j_k(\eta))$ and $\mu^s = \cube_i(\sym_k(\zeta), \sym_k(\eta) )$ are also generators. It is easy to see that 
$
\refl_k^j(\gamma) = \glue_Q(\mu^{r_j} )
$ and 
$
\sym(\gamma) = \glue_Q(\mu^s)$.

Recall that we are assuming that $|Q| \geq 2$. Therefore, there exists $i \in Q$. Evidently, $\faces_i(\Delta')$ is a congruence. So, to finish the proof we need to show that 
$
\faces_k^0(\Delta')
$
is a $(|S|-1)$-dimensional congruence, for each $k \in S$. Here we proceed by induction, supposing that the theorem we are proving is true for indexing sets with cardinality strictly less than $S$. 

Take $k \in S$. 
If $k \in Q$, we compute

\begin{align*}
\faces_k^0(\Delta') &=
\faces_k^0(\glue_Q(\Delta( \{\alpha_i\}_{i \in Q}))) \\
&= \glue_{Q \setminus \{k\}}(\faces_k^0(\Delta(\{\alpha_i\}_{i \in Q})))\\
&= \glue_{Q \setminus \{k\}}(\Delta(\{\alpha_i\}_{i \in Q\setminus \{k\}}))\\
&=_{\text{(inductive assumption)}} \Delta(\{\theta_i\}_{i \in Q \setminus \{k\}}).
\end{align*}
Now suppose $k \in S \setminus Q$.  For each $i \in Q$, set 
$\beta_i = 
\faces_i
\bigg(
\Delta
\left(
\{\theta_j\}_{j \in \{i\} \union S \setminus 
(Q \union \{k\})}
\right)
\bigg)
$
Abusing notation slightly, we have that $\beta_i = \faces_k^0(\alpha_i)$. A straightforward induction on the number of transitive closures taken over $Q$ of $\glue_Q(M(\{\alpha_i\}_{i\in Q}))$ shows that 
\begin{align*}
\faces_k^0(\Delta') &= \faces_k^0(\glue_Q(\Delta( \{\alpha_i\}_{i \in Q})))\\
&= \glue_Q(\Delta( \{\faces_k^0(\alpha_i)\}_{i \in Q}))\\
&= \glue_Q(\Delta( \{\beta_i\}_{i \in Q}))\\
&=_{\text{(inductive assumption)}} \Delta(\{\theta_i\}_{i \in Q \setminus \{k\}}).
\end{align*}

\end{proof}

\section{Higher dimensional Kiss terms}\label{sec:hdkt}

In \cite{threeremarks}, Kiss demonstrates that every modular variety $\var$ has a certain $4$-ary term which one may use to obtain from an arbitrary element of $\rect(\alpha, \beta)$ an element of $\Delta(\alpha, \beta)$, for any $\A \in \var$ and $\alpha, \beta \in\Con(\A)$. A Kiss term therefore has a close connection to the behaviour of the modular binary commutator. In this section we derive from the Kiss term a sequence of `higher dimensional' Kiss terms, each of which has a close connection to the behavior of a higher arity commutator. 

\begin{defn}[see 3.3 of \cite{threeremarks}]\label{def:4differenceterm} 
A four variable term $q_2$ is said to be a $(2)$-dimensional Kiss term for a variety $\var$ if 
\begin{enumerate}
\item $q_2(x,x,y,y) \approx q_2(x,y,x,y) \approx y$ holds in $\var$, and 
\item for all $\A \in \var$ and $\alpha, \beta \in \Con(\A)$, if 
$
\Square[a][c][b][d][.5], \Square[a][c][b][d'][.5] \in \rect(\alpha, \beta),
$
then $\langle q_2(a,b,c,d), q_2(a,b,c,d') \rangle \in [ \alpha, \beta]_{TC} $.
\end{enumerate}
\end{defn}

The following proposition states an important feature of a $(2)$-dimensional Kiss term for a modular variety. 
\begin{prop}[see 3.8 of \cite{threeremarks}]\label{prop:twokisscompletes}
Let $\var$ be a modular variety with a $(2)$-dimensional Kiss term $q_2$. Let $\A \in \var$ and $\alpha, \beta \in \Con(\A)$. If $\Square[a][c][b][d][.5] \in \rect(\alpha, \beta)$, then 
\[
\Square[a][c][b][q_2(a,b,c,d)] \in \Delta(\alpha, \beta).
\]
\end{prop}

Before we go any further, we point out again that the definitions of $\Delta(\alpha, \beta)$ given in \cite{FM} and \cite{taylorsupnil} coincide in the binary modular case. Our use of the higher dimensional Kiss terms makes frequent use of the fact that $\Delta(\alpha, \beta)$ is a $(2)$-dimensional congruence. In particular, the construction relies on the following lemma. 

\begin{lem}\label{lem:twokissmasterlemma}
Let $\var$ be a modular variety with a $(2)$-dimensional Kiss term $q_2$. Let $\A \in \var$ and let $\alpha, \beta \in \Con(\A)$. If $x,y,u,v,x',y',v' \in A$ are such that 
\[
\Square[x][u][y][v][.6] \in \rect(\alpha, \beta)
\qquad \text{and} \qquad
\Square[x][u][y'][v'][.6] \in \Delta(\alpha, \beta),
\]

\[
\text{then }
\begin{tikzpicture}
    [ baseline=(center.base), font=\small,
      every node/.style={inner sep=0.25em}, scale=1 ]
    \node at (0.5,-0.5) (center) {\phantom{$\cdot$}}; 
    \path (0,0)   node (nw) {$u$}
      -- ++(1.5,0)  node (ne) {$q_2(y',y,v',v)$}
      -- ++(0,-1) node (se) {$y$}
      -- ++(-1.5,0) node (sw) {$x$};
    \draw (nw) -- (ne) -- (se) -- (sw) -- (nw);
    
    \end{tikzpicture}
\in \Delta(\alpha, \beta).
\]
\end{lem}

\begin{proof}
The assumptions imply that 
\[
\Square[y'][v'][y][v][.6] \in \rect(\alpha, \beta),
\]
so it follows from Proposition \ref{prop:twokisscompletes} that 
\[
\begin{tikzpicture}
    [ baseline=(center.base), font=\small,
      every node/.style={inner sep=0.25em}, scale=1 ]
    \node at (0.5,-0.5) (center) {\phantom{$\cdot$}}; 
    \path (0,0)   node (nw) {$v'$}
      -- ++(1.5,0)  node (ne) {$p(y',y,v',v)$}
      -- ++(0,-1) node (se) {$y$}
      -- ++(-1.5,0) node (sw) {$y'$};
    \draw (nw) -- (ne) -- (se) -- (sw) -- (nw);
    
    \end{tikzpicture}
\in \Delta(\alpha, \beta).
\]
Because $\Delta(\alpha, \beta)$ is $(2)$-transitive, the result follows.
\end{proof}

\begin{defn}[Higher dimensional Kiss terms]\label{def:higherkissterms}
Let $\var$ be a modular variety with a $(2)$-dimensional Kiss term $q_2$. For each $n \geq 3$, recursively define a $2^n$-ary term $q_n$ as
\[
q_n \coloneqq q_2 \big( 
q_{n-1}(x_0, \dots, x_{2^{n-1}-1}), x_{2^{n-1}-1},
q_{n-1}( x_{2^{n-1}}, \dots, x_{2^n -1}) , x_{2^n -1} 
\big)
\]
\end{defn}

We order the functions belonging to $2^n$ colexicographically in what follows. That is, we write $f < g$ if $f_i < g_i$, where $i$ is the greatest input on which $f$ and $g$ differ. 
The remainder of this section is devoted to proving the following theorem. 

\begin{thm}\label{thm:higherkisscompletes}
Let $\var$ be a modular variety with higher dimensional Kiss terms $q_n$ for $n \geq 2$. Let $\A \in \var$, $n \geq 2$, and $(\theta_0, \dots, \theta_{n-1}) \in \Con(\A)^n$. For $\gamma \in A^{2^n}$, let $\widehat{\gamma} \in A^{2^n}$ be defined as 
\[
\widehat{\gamma}_f = 
\begin{cases}\gamma_f \text{ if } f \neq \textbf{1}\\
q_n(\gamma_\textbf{0}, \dots, \gamma_f, \dots , \gamma_\textbf{1}) \text{ if } 
f = \textbf{1},
\end{cases}
\] 
where the $\gamma_f$ are ordered according to the colexicographical ordering on $2^n$.

If $\gamma \in A^{2^n}$ has the property that $\faces_i^j(\gamma) \in 
\Delta
\left( \{\theta_k\}_{k \in n\setminus \{i\}}
\right) $ for every $i \in n$ and $j \in 2$, then $\widehat{\gamma} \in \Delta(\theta_0, \dots, \theta_{n-1})$.
\end{thm}

\begin{proof}[Proof when $n=3$]

Notice that the $n=2$ case of Theorem \ref{thm:higherkisscompletes} is handled by Proposition \ref{prop:twokisscompletes}. The proof of this special case is included to illustrate the proof of the general case. So, suppose that we are in an $n=3$ situation and that

\[
\gamma  = \Cube[a][c][b][d][e][g][f][h] \in A^{2^3}
\]
satisfies the condition in the theorem statement.

Set $\beta_0 = \faces_0(\Delta(\theta_0, \theta_2))$ and $\beta_1= \faces_1(\Delta(\theta_1, \theta_2))$. Our assumption about $\gamma$ implies that 

\[
\cut_{\{0,1\}} (\gamma) = 
\begin{tikzpicture}
    [ baseline=(center.base), font=\small,
      every node/.style={inner sep=0.02em}, scale=1.2 ]
    \node at (0.5,-0.5) (center) {\phantom{$\cdot$}}; 
    \path (0,0)   node (1nwA) {$c$}
      -- ++(1.3,0)  node (1neA) {$d$}
      -- ++(0,-1) node (1seA) {$b$}
      -- ++(-1.3,0) node (1swA) {$a$};
    \path (.5,-.5)   node (0nwA) {$g$}
      -- ++(1.3,0)  node (0neA) {$h$}
      -- ++(0,-1) node (0seA) {$f$}
      -- ++(-1.3,0) node (0swA) {$e$};
     \draw (0nwA) -- node (c) {} (1nwA);
     \draw (0swA) -- node (d) {}(1swA);
     \draw (0neA) -- node (a) {} (1neA);
     \draw (0seA) -- node (b) {} (1seA);
     \draw (a) -- (b) -- (d) -- (c) -- (a);
    \end{tikzpicture}
\in \rect(\beta_0, \beta_1).
\]
We apply Lemma \ref{lem:twokissmasterlemma} and obtain 
$
\begin{tikzpicture}
    [ baseline=(center.base), font=\small,
      every node/.style={inner sep=0.02em}, scale=1.2 ]
    \node at (0.5,-0.5) (center) {\phantom{$\cdot$}}; 
    \path (0,0)   node (1nwA) {$c$}
      -- ++(1.3,0)  node (1neA) {$q_2(a,b,c,d)$}
      -- ++(0,-1) node (1seA) {$b$}
      -- ++(-1.3,0) node (1swA) {$a$};
    \path (.5,-.5)   node (0nwA) {$g$}
      -- ++(1.3,0)  node (0neA) {$q_2(e,f,g,h)$}
      -- ++(0,-1) node (0seA) {$f$}
      -- ++(-1.3,0) node (0swA) {$e$};
     \draw (0nwA) -- node (c) {} (1nwA);
     \draw (0swA) -- node (d) {}(1swA);
     \draw (0neA) -- node (a) {} (1neA);
     \draw (0seA) -- node (b) {} (1seA);
     \draw (a) -- (b) -- (d) -- (c) -- (a);
    \end{tikzpicture}
\in \Delta(\beta_0, \beta_1).
$
Theorem \ref{thm:deltaofdeltaisdelta} indicates that 
$
\glue_{\{0,1\}}\left(\Delta(\beta_0, \beta_1)\right) = 
\Delta(\theta_0, \theta_1, \theta_2),
$
so we conclude that 
\[
\gamma' = 
\Cube[a][c][b][q_2(a,b,c,d)][e][g][f][q_2(e,f,g,h)][1.2]
\in \Delta(\theta_0, \theta_1, \theta_2).
\]

Set $\alpha_1= \faces_1\left(\Delta(\theta_0, \theta_1)\right)$ and $\alpha_2= \faces_2\left(\Delta(\theta_0,\theta_2)\right)$. Applying the mapping $\cut_{\{1,2\}}$ to $\gamma$ and $\gamma'$ respectively produces the following labeled squares, each belonging to the universe of the indicated algebra:
\[ 
\begin{tikzpicture}
    [ baseline=(center.base), font=\small,
      every node/.style={inner sep=0.02em}, scale=1.2 ]
    \node at (0.5,-0.5) (center) {\phantom{$\cdot$}}; 
    \path (0,0)   node (1nwA) {$c$}
      -- ++(1.3,0)  node (1neA) {$d$}
      -- ++(0,-1) node (1seA) {$b$}
      -- ++(-1.3,0) node (1swA) {$a$};
    \path (.5,-.5)   node (0nwA) {$g$}
      -- ++(1.3,0)  node (0neA) {$h$}
      -- ++(0,-1) node (0seA) {$f$}
      -- ++(-1.3,0) node (0swA) {$e$};
    \node at (.65,0) (a) {\bf{y}};
    \node at (.65,-1) (b) {\bf{x}};
    \node at (1.15,-.5) (c) {\bf{v}};
    \node at (1.15,-1.5) (d) {\bf{u}};
     \draw (0nwA) -- (c) -- (0neA);
     \draw (0swA) -- (d) -- (0seA);
     \draw (1nwA) -- (a) -- (1neA);
     \draw (1swA) -- (b) -- (1seA);
     \draw (a) -- (b) -- (d) -- (c) -- (a);
    \end{tikzpicture}
\in \rect
\left(
\alpha_1, \alpha_2
\right)
\text{ and }
\begin{tikzpicture}
    [ baseline=(center.base), font=\small,
      every node/.style={inner sep=0.02em}, scale=1.2 ]
    \node at (0.5,-0.5) (center) {\phantom{$\cdot$}}; 
    \path (0,0)   node (1nwA) {$c$}
      -- ++(1.6,0)  node (1neA) {$q_2(a,b,c,d)$}
      -- ++(0,-1) node (1seA) {$b$}
      -- ++(-1.6,0) node (1swA) {$a$};
    \path (.5,-.5)   node (0nwA) {$g$}
      -- ++(1.6,0)  node (0neA) {$q_2(e,f,g,h)$}
      -- ++(0,-1) node (0seA) {$f$}
      -- ++(-1.6,0) node (0swA) {$e$};
    \node at (.7,0) (a) {\bf{y'}};
    \node at (.7,-1) (b) {\bf{x}};
    \node at (1.2,-.5) (c) {\bf{v'}};
    \node at (1.2,-1.5) (d) {\bf{u}};
     \draw (0nwA) -- (c) -- (0neA);
     \draw (0swA) -- (d) -- (0seA);
     \draw (1nwA) -- (a) -- (1neA);
     \draw (1swA) -- (b) -- (1seA);
     \draw (a) -- (b) -- (d) -- (c) -- (a);
    \end{tikzpicture}
\in \Delta(\alpha_1, \alpha_2).
\]

Indeed, the truth of the first membership follows from our assumption about $\gamma$ and the truth of the second membership follows from Theorem \ref{thm:deltaofdeltaisdelta}. The vertex labels of the above labeled squares have been named so that we may apply Lemma \ref{lem:twokissmasterlemma}, and we conclude that
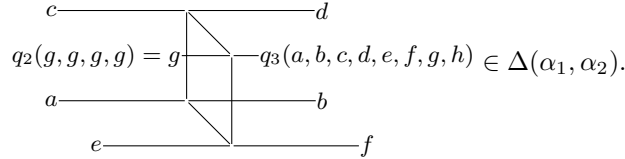
\[
\begin{tikzpicture}
    [ baseline=(center.base), font=\small,
      every node/.style={inner sep=0.02em}, scale=1.2 ]
    \node at (0.5,-0.5) (center) {\phantom{$\cdot$}}; 
    \path (0,.1)   node (1nwA) {$c$}
      -- ++(3,0)  node (1neA) {$d$}
      -- ++(0,-1) node (1seA) {$b$}
      -- ++(-3,0) node (1swA) {$a$};
    \path (.5,-.4)   node (0nwA) {$q_2(g,g,g,g)=g$}
      -- ++(3,0)  node (0neA) {$q_3(a,b,c,d,e,f,g,h)$}
      -- ++(0,-1) node (0seA) {$f$}
      -- ++(-3,0) node (0swA) {$e$};
    \node at (1.5,.1) (a) {};
    \node at (1.5,-.9) (b) {};
    \node at (2,-.4) (c) {};
    \node at (2,-1.4) (d) {};
     \draw (0nwA) -- (c) -- (0neA);
     \draw (0swA) -- (d) -- (0seA);
     \draw (1nwA) -- (a) -- (1neA);
     \draw (1swA) -- (b) -- (1seA);
     \draw (a) -- (b) -- (d) -- (c) -- (a);
    \end{tikzpicture}
 \in \Delta (\alpha_1, \alpha_2 ).
\]
A final application of Theorem \ref{thm:deltaofdeltaisdelta} finishes the proof of Theorem \ref{thm:higherkisscompletes} for the case $n=3$.

\end{proof}

\begin{proof}[Proof of general case]
Let $n \geq 3$ and suppose that Theorem \ref{thm:higherkisscompletes} holds for $q_{n-1}$. Let $\gamma \in A^{2^n}$ satisfy the conditions of the theorem statement. For each $i \in n-1$, set $\beta_i = \faces_i(\Delta(\theta_i, \theta_{n-1}))$. Let $i \in n-1$. We assume that 
$
\faces_i^j(\gamma) \in \Delta\left( \{\theta_k\}_{k \in n\setminus \{i\}} \right),
$
for each $j \in 2$.
We apply the $\lines_{n-1}$ map to each side of this and conclude that
\[
\lines_{n-1}\left(
\faces_i^j(\gamma)
\right)
=
\faces_i^j
\left(
\lines_{n-1}(\gamma)
\right)
\in \lines_{n-1}
\left(
\Delta\left( \{\theta_k\}_{k \in n\setminus \{i\}} \right)
\right).
\]
We now use Theorem \ref{thm:deltaofdeltaisdelta} and the fact that $\glue_{n-1}$ is the inverse of $\lines_{n-1}$ to see that
\begin{align*}
\lines_{n-1}
\left(
\Delta\left( \{\theta_k\}_{k \in n\setminus \{i\}} \right)
\right)
&= \lines_{n-1}
\left(
\glue_{n-1}
\left( 
\Delta
\left( \{\beta_k\}_{k \in n-1\setminus \{i\}}
\right)
\right)
\right)\\
&= 
\Delta
\left( \{\beta_k\}_{k \in n-1\setminus \{i\}}
\right). 
\end{align*}
Set $\mu = \lines_{n-1}(\gamma)$. We have just shown that $\mu$ satisfies the assumption of the theorem we are proving for the lower dimension of $n-1$, and we inductively assume the theorem to hold in this case. There therefore exists $\widehat{\mu}$ as in the theorem statement such that
$
\widehat{\mu} \in \Delta
\left( \{\beta_k\}_{k \in n-1}
\right).
$
Set $\gamma' = \glue_{n-1}(\widehat{\mu})$. We apply Theorem \ref{thm:deltaofdeltaisdelta} once again and conclude that 
$
\gamma' \in \Delta(\theta_0, \dots, \theta_{n-1}).
$

It follows from the definition of $\widehat{\mu}$ that 
\[
\gamma'_f= 
\begin{cases}
 q_{n-1}
\big(
\underbrace{
\faces_{n-1}^0(\gamma)_\textbf{0}, \dots, \faces_{n-1}^0(\gamma)_g, \dots,
\faces_{n-1}^0(\gamma)_\textbf{1}
}_{\text{colex ordering on $g \in 2^{n-1}$}}
\big) 
\text{ if } f = (1, \dots, 1,0), \\
 q_{n-1}
\big(
\underbrace{
\faces_{n-1}^1(\gamma)_\textbf{0}, \dots, \faces_{n-1}^0(\gamma)_g, \dots,
\faces_{n-1}^1(\gamma)_\textbf{1}
}_{\text{colex ordering on $g \in 2^{n-1}$}}
\big) 
\text{ if } f=(1, \dots, 1,1), \\
\gamma_f \text{ otherwise.}
\end{cases}
\]
Set $\zeta = \cut_{\{n-2, n-1\}}(\gamma) $ and $\zeta' = \cut_{\{n-2,n-1\}}(\gamma').$
We now have 

\[
\zeta= \Square[\textbf{x}][\textbf{u}][\textbf{y}][\textbf{v}]
\text{ and }
\zeta' = \Square[\textbf{x}][\textbf{u}][\textbf{y'}][\textbf{v'}]
\]
where 
\begin{align*}
\textbf{x} &= \cut_{\{n-2, n-1\}}(\gamma)_{\{\langle n-2, 0\rangle, \langle n-1, 0 \rangle \}},
 \\
\textbf{y} &= \cut_{\{n-2, n-1\}}(\gamma)_{\{\langle n-2, 1\rangle, \langle n-1, 0 \rangle \}},
&&\textbf{y'} = \cut_{\{n-2, n-1\}}(\gamma')_{\{\langle n-2, 1\rangle, \langle n-1, 0 \rangle \}}, \\
\textbf{u} &= \cut_{\{n-2, n-1\}}(\gamma)_{\{\langle n-2, 0\rangle, \langle n-1, 1 \rangle \}},  
 \\
\textbf{v} &= \cut_{\{n-2, n-1\}}(\gamma)_{\{\langle n-2, 1\rangle, \langle n-1, 1 \rangle \}}.
 &&\textbf{v'}= \cut_{\{n-2, n-1\}}(\gamma')_{\{\langle n-2, 1\rangle, \langle n-1, 1 \rangle \}}. \\
\end{align*}
Our goal is to apply Lemma \ref{lem:twokissmasterlemma}, with the algebra under consideration set to $\Delta(\theta_0, \dots, \theta_{n-3})$ and the two congruences under consideration set to 
\[
\alpha_{n-2} = \faces_{n-2}(\Delta(\{\theta_j\}_{j\in n\setminus \{n-1\}}))
\text{ and }  
\alpha_{n-1} = \faces_{n-1}(\Delta(\{\theta_j\}_{j\in n\setminus \{n-2\}})).
\] Again, the assumption we made of $\gamma$ implies that $\zeta \in \rect(\alpha_{n-2}, \alpha_{n-1})$, while Theorem \ref{thm:deltaofdeltaisdelta} implies that $\zeta' \in \Delta(\alpha_{n-2}, \alpha_{n-1})$. Having established the conditions of Lemma \ref{lem:twokissmasterlemma}, we conclude that 
\[
\epsilon=
\begin{tikzpicture}
    [ baseline=(center.base), font=\small,
      every node/.style={inner sep=0.25em}, scale=1 ]
    \node at (0.5,-0.5) (center) {\phantom{$\cdot$}}; 
    \path (0,0)   node (nw) {$\textbf{u}$}
      -- ++(1.5,0)  node (ne) {$q_2(\textbf{y'}, \textbf{y}, \textbf{v'}, \textbf{v})$}
      -- ++(0,-1) node (se) {$\textbf{y}$}
      -- ++(-1.5,0) node (sw) {$\textbf{x}$};
    \draw (nw) -- (ne) -- (se) -- (sw) -- (nw);
    
    \end{tikzpicture}
\in \Delta(\alpha_{n-2}, \alpha_{n-2}).
\]

\end{proof}

It is now straightforward to check that, for $g \in 2^{n-2}$, 
\[
q_2(\textbf{y'}, \textbf{y}, \textbf{v'}, \textbf{v})_g =
q_2(\textbf{y'}_g, \textbf{y}_g, \textbf{v'}_g, \textbf{v}_g)=
\begin{cases}
\textbf{v}_g &\text{ if } g\neq \textbf{1}\\
q_n(\underbrace{\gamma_\textbf{0}, \dots, \gamma_f, \dots, \gamma_{\textbf{1}}}_{\text{colex ordering on $f \in 2^{n}$}})
&\text{ if } g  =\textbf{1}.
\end{cases}
\]
Therefore, we set $\widehat{\gamma} = \glue_{\{n-2, n-1\}}(\epsilon)$. It is similarly straightforward to check that $\widehat{\gamma}$ has the property claimed by the theorem, so the proof is finished. 

\section{An application of higher dimensional Kiss terms}\label{sec:application}
In this section we show that, for an algebra $\A$ belonging to a modular variety, the collection of all $\Delta(\theta_0, \dots, \theta_{n-1})$ for all $n\geq 2$ and $(\theta_0, \dots, \theta_{n-1}) \in \Con(\A)^n$ is completely determined by the higher commutator. We established in \cite{taylorsupnil} that this collection of relations completely determines the hypercommutator for $\A$. In view of Theorem \ref{thm:tc=hhcinmodular}, this means that the term condition higher commutator and the collection of all $\Delta$ relations completely determine each other for an algebra with Day terms. After this is established we will follow \cite{orsalrel} and conclude that any collection of clones on some set that share congruences, higher commutators, and Day terms has a greatest element. 

We begin with a lemma which, we wish to note, could be used to provide a cleaner proof of parts of Theorem 4.10 from \cite{taylorsupnil}. Recall from \cite{taylorsupnil} that for an algebra $A$, $n\geq 2$, and $x,y \in A$, the $(n)$-dimensional commutator cube for the pair $\langle x, y\rangle$ is the labeled cube belonging to $A^{2^n}$ such that the pivot vertex is labeled by $y$, while the other vertices are labeled by $x$. This cube is denoted $\com_n(x,y)$.

\begin{lem}\label{lem:deltashifting}
Let $A$ be an algebra, $n\geq 1$, and $R \subseteq A^{2^n}$ an $(n)$-dimensional congruence. Let $\gamma \in R$. If $\com_n(\gamma_\textbf{1}, q) \in R$ for some $q \in A$, then $\widehat{\gamma} \in R$, where 
\[
\widehat{\gamma}_f =
\begin{cases}
q &\text{ if } f = \textbf{1}\\
\gamma_f &\text{ otherwise}.
\end{cases}
\] 
\end{lem}

\begin{proof}
The proof proceeds by induction. The case when $n=1$ is an obvious application of the fact that a $(1)$-dimensional congruence is a transitive relation. So, suppose that $n\geq 2$ and that the lemma holds for $n-1$. Let $R \leq A^{2^n}$, $\gamma \in R$, and suppose that there is a $q \in A$ such that $\com_n(\gamma_{\textbf{1}}, q) \in R$. Suppose that the $(n-1)$-pivot line of $\gamma$ is the pair $\langle z, \gamma_\textbf{1} \rangle$. It follows from Corollary 2.5 of \cite{taylorsupnil} that $\cube_{n-1}(z, \gamma_\textbf{1}) \in R$. Because $R$ is $(n)$-transitive, we conclude that 
\begin{align*}
\mu &= \glue_{n-1}
\left(
\langle
\faces_{n-1}^0(\cube_{n-1}(z,\gamma_\textbf{1})),
\faces_{n-1}^1(\com_n(\gamma_\textbf{1},q))
\rangle
\right)\\
&= \glue_{n-1}(\langle \cube(z) , \com_{n-1}(\gamma_\textbf{1}, q) \rangle) \in R.
\end{align*}
We can describe $\mu$ more concretely, where for $f \in 2^n$
\[
\mu_f = 
\begin{cases}
z &\text{ if } f_{n-1}= 0\\
\gamma_\textbf{1} &\text{ if } f_{n-1} =1 \text{ and } f\neq \textbf{1}\\
q &\text{ if } f = \textbf{1}.
\end{cases}
\]

It follows from Lemma 2.4 of \cite{taylorsupnil} that $\lines_{n-1}(R)$ is a $(n-1)$-dimensional congruence of $\A^2$. Notice that 
$\lines_{n-1}(\mu) = \com_{n-1}(\langle z, \gamma_\textbf{1} \rangle, \langle z, q \rangle )$. So, we can apply inductive assumption to get that $ \epsilon \in \lines_{n-1}(R)$, where, for $g \in n-1$, we have that
\[
\epsilon_g=
\begin{cases}
\lines_{n-1}(\gamma)_g &\text{ if } g \neq \textbf{1}, \\
\langle z, q \rangle &\text{ if } g = \textbf{1}.
\end{cases}
\]
To finish the proof, set $\widehat{\gamma} = \glue_{\{n-1\}}(\epsilon)$.

\end{proof}

\begin{prop}\label{prop:characterizeDelta}
Let $\A$ be an algebra belonging to a modular variety $\var$. Let $n\geq 2$, $q_n$ be an $(n)$-dimensional Kiss term for $\var$, and $(\theta_0, \dots, \theta_{n-1}) \in \Con(\A)^n$. Then, $\gamma \in \Delta(\theta_0, \dots, \theta_{n-1})$ if and only if 
\begin{enumerate}
\item $\faces_i^j(\gamma) \in 
\Delta
\left( \{\theta_k\}_{k \in n\setminus \{i\}}
\right) $ for every $i \in n$ and $j \in 2$, and 
\item $\langle \gamma_\textbf{1} , q_n(\gamma_\textbf{0}, \dots, \gamma_f, \dots, \gamma_{\textbf{1}}) \rangle \in [\theta_0, \dots, \theta_{n-1}]$.
\end{enumerate}
\end{prop}

\begin{proof}
Suppose that $\gamma \in \Delta(\theta_0, \dots, \theta_{n-1})$. It follows from Lemma 2.14 of \cite{taylorsupnil} that $\gamma$ satisfies \emph{(1)}, which also means that $\gamma$ is within the scope of Theorem \ref{thm:higherkisscompletes}. Let $\widehat{\gamma}$ be as in the conclusion of Theorem \ref{thm:higherkisscompletes}.
It follows from the $(n)$-symmetry and $(n)$-transitivity of $\Delta(\theta_0, \dots, \theta_{n-1})$ that 
\[\mu = \glue_{\{n-1\}}
\left(
\langle \faces_{n-1}^1(\gamma), \faces_{n-1}^1(\widehat{\gamma}) \rangle 
\right)
\in \Delta(\theta_0, \dots, \theta_{n-1}).
\]
Because the pair $\langle \gamma_\textbf{1} , q_n(\gamma_\textbf{0}, \dots, \gamma_f, \dots, \gamma_{\textbf{1}} )\rangle$ is $(n-1)$-supported by $\mu$, we may apply Theorem 4.10 of \cite{taylorsupnil} along with \ref{thm:tc=hhcinmodular} and conclude that $\langle \gamma_\textbf{1} , q_n(\gamma_\textbf{0}, \dots, \gamma_f, \dots, \gamma_{\textbf{1}} )\rangle \in [\theta_0, \dots, \theta_{n-1}]$. This establishes \emph{(2)}.

Now, suppose that \emph{(1)} and \emph{(2)} hold. We may again apply Theorem \ref{thm:higherkisscompletes} and obtain $\widehat{\gamma} \in \Delta(\theta_0, \dots, \theta_{n-1})$. Next, we use \emph{(2)} along with Theorem 4.10 of \cite{taylorsupnil} to conclude that $\com_n( q_n(\gamma_\textbf{0}, \dots, \gamma_f, \dots, \gamma_{\textbf{1}}) ,\gamma_\textbf{1} ) \in \Delta(\theta_0, \dots, \theta_{n-1})$. We apply Lemma \ref{lem:deltashifting} and conclude that $\gamma \in \Delta(\theta_0, \dots, \theta_{n-1})$.
\end{proof}

In what follows we need to distinguish between commutators computed in different algebras, provided it makes sense to do so. Let $\A$ and $\mathbb{B}$ be algebras with a common universe and suppose that $\{\theta_i\}_{i \in S} \subseteq \Con(\A)\cap \Con(\mathbb{B})$ for some $S \finsub \nat$. We denote the $\Delta(\{\theta_i\}_{i \in S})$ and $[\{\theta_i\}_{i \in S}]$ computed in either $ \A$ or $\mathbb{B}$ with a subscript, i.e.\ $\Delta_\A( \{\theta_i\}_{i \in S})$ and $[\{\theta_i\}_{i \in S}]_\A$ are these objects computed in $\A$, while $\Delta_\mathbb{B}( \{\theta_i\}_{i \in S})$ and $[\{\theta_i\}_{i \in S}]_\mathbb{B}$ are computed in $\mathbb{B}$. We make no effort to distinguish between the term condition and hypercommutator here, because each algebra we work with belongs to a modular variety. 

\begin{cor}[cf.\ Corollary 4.3 of \cite{orsalrel}]\label{cor:shareddelta}
If $\A$ and $\mathbb{B}$ are algebras that share a universe, Day terms $p_0, \dots, p_k$, and congruences $\{\theta_i\}_{i \in S}$ for some $S \finsub \nat$, then 
$[\{\theta_i\}_{i \in T}]_\A = [\{\theta_i\}_{i \in T}]_\mathbb{B}$ for all $T \subseteq S$
if and only if
$\Delta_\A( \{\theta_i\}_{i \in S}) = \Delta_\mathbb{B}( \{\theta_i\}_{i \in S})$. 
\end{cor}

\begin{proof}
Notice that, because $\A$ and $\mathbb{B}$ share a choice of Day terms, they also share a sequence $q_2, \dots, q_n, \dots $ of higher dimensional Kiss terms. For the forward implication we proceed by induction on the size of $S$. The statement is trivial if $|S| =1$. Let $n\geq 2$ and assume it holds for $|S| =n-1$. Take $\gamma \in \Delta_\A(\{\theta_i\}_{i \in S})$. Then $\gamma$ satisfies \emph{(1)} and \emph{(2)} of Proposition \ref{prop:characterizeDelta}. Applying the inductive assumption, the assumption that $[\{\theta_i\}_{i \in T}]_\A = [\{\theta_i\}_{i \in T}]_\mathbb{B}$, and the fact that these algebras share the term $q_n$ allows us to conclude that $\gamma$ satisfies \emph{(1)} and \emph{(2)}, now for the algebra $\mathbb{B}$. This shows that $\gamma \in \Delta_\mathbb{B}(\{\theta_i\}_{i \in S})$. So, $\Delta_\A(\{\theta_i\}_{i \in S}) \subseteq \Delta_\mathbb{B}(\{\theta_i\}_{i \in S})$. The other containment is proved in an identical manner. 

The other direction follows from Lemma 2.14 and Theorem 4.10 of \cite{taylorsupnil}.
\end{proof}

\begin{thm}[cf.\ Theorem 1.3 of \cite{orsalrel}]\label{thm:comparecommutators}
Let $\A$ be an algebra with Day terms $m_0, \dots, m_k$. Suppose $S \finsub \nat$ and let $\{\theta_i\}_{i \in S} \subseteq \Con(\A)$. There exists a greatest clone $\mathcal{C}$ on $A$ satisfying
\begin{enumerate}
\item $m_0, \dots, m_k \in \mathcal{C} $,
\item $\mathcal{C}$ preserves $\{\theta_i\}_{i \in S}$, and
\item $[\{\theta_i\}_{i \in T}]_\A = [\{\theta_i\}_{i \in T}]_{\langle A; \mathcal{C}\rangle}$ for every $T \subseteq S$.
\end{enumerate}
\end{thm}

\begin{proof}
Let $\mathcal{C}$ be the collection of all polymorphisms of $\Delta_\A(\{\theta_i\}_{i \in S}$. We first show that \emph{(1)}, \emph{(2)} and \emph{(3)} hold. It is immediate that \emph{(1)} holds. It follows from Lemma 2.14 of \cite{taylorsupnil} that $\Delta_\A(\{\theta_i\}_{i \in T})$ is definable from $\Delta_\A(\{\theta_i\}_{i \in S})$ by a positive primitive formula, for every $T\subseteq S$. Therefore, $\mathcal{C}$ preserves all of these lower dimensional relations as well. In particular, \emph{(2)} holds. Because $\clo(\A) \subseteq \mathcal{C}$, we have that $\Delta_\A(\{\theta_i\}_{i \in S}) \subseteq \Delta_{\langle A; \mathcal{C}\rangle } (\{\theta_i\}_{i \in S})$. Because $\Delta_\A (\{\theta_i\}_{i \in S})$ is preserved by $\mathcal{C}$ and is a higher dimensional equivalence relation containing the generators of $ \Delta_{\langle A; \mathcal{C}\rangle } (\{\theta_i\}_{i \in S})$, we see that the two relations are equal. We apply Corollary \ref{cor:shareddelta} and conclude that \emph{(3)} holds. 

Now we show that $\mathcal{C}$ is the greatest such clone. Suppose that $\mathcal{D}$ also satisfies \emph{(1)}, \emph{(2)}, and \emph{(3)}. We apply Corollary \ref{cor:shareddelta} to the algebras $\langle A; \mathcal{C}\rangle$ and $\langle A; \mathcal{D} \rangle$ and conclude that $\Delta_{\langle A; \mathcal{C}\rangle } (\{\theta_i\}_{i \in S}) = \Delta_{\langle A; \mathcal{D}\rangle } (\{\theta_i\}_{i \in S})$. Therefore, every operation belonging to $\mathcal{D}$ preserves $\Delta_\A(\{\theta_i\}_{i \in S})$, hence $\mathcal{D} \subseteq \mathcal{C}$.
\end{proof}

\begin{cor}[cf.\ Corollary 1.4 of \cite{orsalrel}]
Let $\A$ be an algebra that has Day terms $m_0, \dots, m_k$. There exists a largest clone $\mathcal{C}$ on $A$ satisfying 
\begin{enumerate}
\item $p_0, \dots, p_k \in \mathcal{C}$ 
\item $\langle A; \mathcal{C} \rangle$ and $\A$ have the same congruences, and
\item $\langle A; \mathcal{C} \rangle$ and $\A$ have the same higher commutator operations.
\end{enumerate} 
\end{cor}
\begin{proof}
The clone $\mathcal{C}$ is the intersection of all possible clones guaranteed by Theorem \ref{thm:comparecommutators}.
\end{proof}

\bibliographystyle{amsplain}   
\bibliography{refs.bib}
\begin{center}
  \rule{0.61803\textwidth}{0.1ex}   
\end{center}
\subjclass{MSC 08A40 (08A05, 08B05)}
\end{document}